\documentclass{ws-m3as-ppn}

\def\@xthm#1#2{\@beginassumption{#2}{\csname the#1\endcsname}{}\ignorespaces}
\def\@ythm#1#2[#3]{\@opargbeginassumption{#2}{\csname the#1\endcsname}{#3}\ignorespaces}

\def\@beginassumption#1#2#3{\par\addvspace{8pt plus3pt minus2pt}
\noindent{\csname#1headfont\endcsname#1\ \ignorespaces#3 #2.}
\csname#1font\endcsname\hskip.5em\ignorespaces}
\def\@endassumption{\par\addvspace{8pt plus3pt minus2pt}\@endparenv}

\usepackage{color}
\usepackage{hyperref}
\usepackage{amssymb}
\usepackage{amsfonts}
\usepackage{amsmath}
\usepackage{latexsym}
\usepackage[spanish,french,english]{babel}
\usepackage[utf8]{inputenc}

\begin{document}

\definecolor{darkgreen}{rgb}{0,0.4,0}
\definecolor{darkred}{rgb}{0.8,0,0}
\definecolor{darkblue}{rgb}{0,0,0.7}
\definecolor{braun}{rgb}{0.5,0.3,0}
\definecolor{gelb}{rgb}{1,1,0}
\definecolor{blaugrau}{rgb}{0,0.5,0.5}
\definecolor{hellblau}{rgb}{0,1,1}
\definecolor{violet}{rgb}{0.3,0,0.3}
\definecolor{lila}{rgb}{1,0,1}
\definecolor{grau}{rgb}{0.3,0.3,0.3}
\newcommand{\Cb}[1]{\textcolor{blue}{#1}}
\newcommand{\Cg}[1]{\textcolor{green}{#1}}
\newcommand{\Cr}[1]{\textcolor{red}{#1}}
\newcommand{\Cc}[1]{\textcolor{cyan}{#1}}
\newcommand{\Cm}[1]{\textcolor{magenta}{#1}}
\newcommand{\Cy}[1]{\textcolor{yellow}{#1}}
\newcommand{\Cv}[1]{\textcolor{violet}{#1}}
\newcommand{\Cdb}[1]{\textcolor{darkblue}{#1}}
\newcommand{\Cdr}[1]{\textcolor{darkred}{#1}}
\newcommand{\Cdg}[1]{\textcolor{darkgreen}{#1}}
\newcommand{\Cbraun}[1]{\textcolor{braun}{#1}}

\newcommand{\red}{\color{red}}
\newcommand{\nc}{\normalcolor}

\newcommand{\Huno}{\mathrm H^1(\Real^3)}
\newcommand{\Hdos}{\mathrm H^2(\Real^3)}
\newcommand{\Real}{\mathbb R}
\newcommand{\Nat}{\mathbb N}
\newcommand{\Lq}{\L^{q}}
\newcommand{\Lp}{\L^{p}(\Real^3)}
\newcommand{\Ldos}{\L^2(\Real^3)}
\newcommand{\Lqp}{\L^{pq}}
\newcommand{\LqpT}{\L_{\mathrm{T}}^{pq}}
\newcommand{\Linfty}{\L^{\infty}}
\newcommand{\Cont}{\mathrm{C}}
\newcommand{\T}{\mathrm{T}}
\newcommand{\X}{\mathrm{X}}
\newcommand{\A}{\mathcal{A}}
\newcommand{\Rep}{\mathrm{Re}}
\newcommand{\Imp}{\mathrm{Im}}

\newcommand{\Email}[1]{{\sl E-mail:\/} {\rm\textsf{#1}}}
\newcommand{\be}[1]{\begin{equation}\label{#1}}
\newcommand{\ee}{\end{equation}}
\newcommand{\nrm}[2]{\|{#1}\|_{\L^{#2}(\Real^3)}}
\newcommand{\ir}[1]{\int_{\Real^3}#1\,dx}
\newcommand{\supp}[1]{\hbox{\rm{supp}}(#1)}
\newcommand{\C}{\mathrm C}
\newcommand{\CGN}{\C_{\mathrm{GN}}}
\newcommand{\D}[1]{\mathrm D[#1]}
\newcommand{\E}{\mathrm E}
\renewcommand{\L}{\mathrm L}

\newcommand{\cadre}[1]{\begin{center}\begin{tabular}{|c|}\hline\cr
#1\cr\cr\hline \end{tabular}\end{center}}

\font\dc=cmbxti10

\overfullrule=0pt
\def\bull{{\vrule height.9ex width.8ex depth.1ex}}
\def\boxitt#1{\vbox{\hrule\hbox{\vrule\kern3pt\vbox{\kern7pt#1\kern7pt}\kern3pt\vrule}\hrule}}

\parindent=20pt

\title{EXISTENCE OF STEADY STATES FOR THE MAXWELL-SCHR\"ODINGER-POISSON SYSTEM:\\
EXPLORING THE APPLICABILITY OF THE CONCENTRATION-COMPACTNESS PRINCIPLE\footnote{~This work has been partially supported by MINECO (Spain), Project MTM2011-23384, and ANR projects NoNAP and STAB (France).}}

\author{I. CATTO$^{(1)}$, J. DOLBEAULT$^{(1)}$, O. S\'ANCHEZ$^{(2)}$ and J. SOLER$^{(2)}$}
\vskip1truecm

\address{$^{(1)}$Ceremade (UMR CNRS no. 7534), Universit\'e Paris-Dauphine\\
Place de Lattre de Tassigny, F-75775 Paris Cedex~16, France\\
\Email{catto@ceremade.dauphine.fr, dolbeaul@ceremade.dauphine.fr}\\[6pt]
$^{(2)}$Departamento de Matem\'atica Aplicada,\\
University of Granada, 18071--Granada, Spain\\
\Email{ossanche@ugr.es, jsoler@ugr.es}}
\medskip\maketitle\thispagestyle{empty}
\begin{history}
\end{history}
\begin{abstract} This paper is intended to review recent results and open problems concerning the existence of steady states to the Maxwell-Schr\"odinger system. A combination of tools, proofs and results are presented in the framework of the concentration--compactness method.\end{abstract}



\parindent=20pt

\section{Introduction}

The concentration--compactness method is nowadays a basic tool in applied mathematics for the analysis of variational problems with lack of compactness or more specifically for proving existence of solutions of non-linear partial differential equations which are invariant under a group of transformations. In this review we explore the applicability of the concentration--compactness method on the $X^\alpha$-Schr\"odinger-Poisson model. We will also highlight some related questions, which raise a number of open issues. 

Our purpose is to study the existence of steady states of the so-called $X^\alpha$-Schr\"odinger-Poisson ($X^\alpha$-SP) model or Maxwell-Schr\"odinger-Poisson system:
\begin{eqnarray}
&&i\,\frac{\partial \psi}{\partial t}=-\Delta_x\psi+V(x,t)\,\psi-C\,|\psi(x,t)|^{2\alpha}\,\psi\,,\nonumber\\
&&-\Delta_xV=\epsilon\,4\pi\,|\psi|^2,\label{XASP}\\
&&\psi(x,t=0)=\phi(x)\,,\nonumber
\end{eqnarray}
with $\phi\in \L^2(\Real^3)$, $x\in\Real^3$, $t\ge 0$. The self-consistent Poisson potential $V$ is explicitly given by $V(x,t)=\epsilon\,|\psi(x,t)|^2\star|x|^{-1}$, where $\star$ refers to the convolution with respect to $x$ on $\Real^3$ and where $\epsilon$ takes the value $+1$ or $-1$, depending whether the interaction between the particles is repulsive or attractive. The system \eqref{XASP} can therefore be reduced to a single non-linear and non-local Schr\"odinger-type equation
\begin{eqnarray}\label{eq:XASP}
&&i\,\frac{\partial \psi}{\partial t}=-\Delta_x\psi+\epsilon\,\Big(|\psi|^2 \star|x|^{-1}\Big)\,\psi-C\,|\psi|^{2\alpha}\,\psi\,,\\
&&\psi(x,t=0)=\phi(x)\,.\nonumber
\end{eqnarray}
Such a model appears in various frameworks, such as black holes in gravitation \hbox{($\epsilon=-1$)}~\cite{RuSo}, one-dimensional reduction of electron density in plasma physics \hbox{($\epsilon=+1$)}, as well as in semiconductor theory ($\epsilon=+1$), as a correction to the Schr\"odinger-Poisson system (which is $X^\alpha$-SP with $C=0$): see~\cite{BoLoSo,LiSi,Mauser} and references therein.

In the plasma physics case, the $X^\alpha$-SP correction takes into account a nonlinear, although local, correction to the Poisson potential
of opposite sign given by $-\,C\,\vert\psi\vert^{2\alpha}$, where $C$ is a positive constant and where the parameter $\alpha$, responsible for the name of the model, takes values in the range $0<\alpha\le\frac23$. Some relevant values are for example $\alpha=\frac13$, which is called the Slater correction, or $\alpha=\frac23$, which gives rise to the so-called Dirac correction. The idea is to balance the Poisson potential (also called Coulombian potential in the electrostatic case) with a local potential term of opposite sign. This generates a competition between the two potential energies and the kinetic energy that, depending on the values of the constant $C$, can modify the typically dispersive dynamics of the Schr\"{o}dinger-Poisson system~\cite{IlSwLa,SaSo1} in the plasma physics case. The local nonlinear term also modifies the properties of the solutions in the gravitational case, thus leading to a richer behaviour~\cite{BoLoSaSo}. Note that the physical constants have been normalized to unity here for the sake of simplicity.

Throughout the paper we focus our attention on the plasma physical case. Similar techniques can be used for extending our results to the gravitational case. Notice that when $\epsilon=-1$ (gravitational case), the sign of the energy associated to the Poisson potential (also called Newtonian potential) allows to introduce symmetric rearrangements that contribute to simplify some computations~\cite{Lieb,LiebLoss}. In this paper, we shall therefore assume that
\[
\epsilon=+1\,.
\]

We will be concerned with the existence of standing waves, that is, solutions to \eqref{eq:XASP} of the form
\begin{eqnarray*}
\psi(x,t)=e^{i\ell_M t}\,\varphi(x)
\end{eqnarray*}
with $\ell_M>0$ and $\varphi$ in $\L^2(\Real^3)$ solving
\be{eq:sw}
-\Delta \varphi+\epsilon\,\big(|\varphi|^2 \star|x|^{-1}\big)\,\varphi-C\,|\varphi|^{2\,\alpha}\,\varphi\,+\ell_M\,\varphi=0\,.
\ee
Equation~\eqref{eq:sw} is a special case of Schr\"odinger-Maxwell equations~\cite{DAMu}.

The existence and stability analysis of such solutions relies on some preserved physical quantities. The total \emph{mass} (which is also the total electronic charge in the repulsive case, when $\epsilon=+1$)
\[
M[\psi]:=\int_{\Real^3}|\psi(x,t)|^2\,dx
\]
and the \emph{energy} functional
\[
\E[\psi]:=\E_{\mathrm{kin}}[\psi]+\E_{\mathrm{pot}}[\psi]
\]
are invariant quantities for any solution of $X^\alpha$-SP along the time evolution, where the \emph{kinetic} and \emph{potential} energies are defined by
\[
\E_{\mathrm{kin}}[\psi]:=\frac12\ir{|\nabla\psi(x,t)|^2}\,,\quad\E_{\mathrm{pot}}[\psi]:=\frac{\epsilon}4\,\D\psi-\frac C{2\alpha+2} \int_{\Real^3} |\psi(x,t)|^{2\alpha+2}\,dx
\]
and
\[
\D\psi:=\iint_{\Real^3\times\Real^3}\frac{|\psi(x,t)|^2\,|\psi(x',t)|^2}{|x-x'|}\,dx\,dx'\,.
\]

The existence of standing waves has been carried out from various perspectives in the vast mathematical literature devoted to this topic. Either one investigates the existence of critical points of the functional $\E[\varphi]+\ell_M\,M[\varphi]$ on the whole space $\Huno$, with the parameter $\ell_M$ being given and fixed, and in that case the $\L^2(\Real^3)$ norm of the solution is not prescribed (see for instance~\cite{DR} and references therein); or one looks for critical points of the energy functional $\E[\varphi]$ with prescribed $\L^2(\Real^3)$ norm, and then the parameter $\ell_M$ enters into the game as a Lagrange multiplier of the constrained minimization problem. From a physical point of view, the most interesting critical points, the so-called \emph{steady states,} are the minimizers of the problem
\be{minienergy}
I_M:=\inf\big\{\E[\varphi]\,:\,\varphi\in\Sigma_M\big\}\,,\quad\Sigma_M:=\big\{\varphi\in\Huno\,:\,\|\varphi\|_{\Ldos}^2=M\big\}\,.
\ee
Their interest lies in \emph{stability} properties stated in terms of the energy and the mass. Such a feature is of course well known in the literature, see for instance~\cite{CaLi}, and it provides an easier approach than other methods, which are anyway needed when elaborate variational methods are required like in~\cite{BeJeLu}. The energy functional is not bounded from below when $\alpha>\frac23$. When $\alpha>2$, the exponent $2\alpha+2$ lies outside of the interval $(2,6)$ and then $\Huno$ is not embedded in $L^{2\alpha+2}(\Real^3)$. We therefore restrict our analysis to the range $\alpha$ in $(0,2)$.

Concerning the existence of steady states, let us make the following observations. First of all, the energy and mass functionals are translation invariant that is, for every $y\in \Real^3$,
\[
\E[\varphi(\cdot +y)]=\E[\varphi]\,, \quad M[\varphi(\cdot +y)]=M[\varphi]\,.
\]
 Therefore the concentration--compactness method~\cite{bi:PLL-CC1cras,bi:PLL-CC1,bi:PLL-CC2} is the natural framework for the study of the existence of a minimizer and for the analysis of the behavior of the minimizing sequences to \eqref{minienergy} and their possible lack of compactness. According to the terminology of the concentration--compactness principle, from any minimizing sequence $\{\varphi_n\}_{n\ge 1}$ in $\Sigma_M$ we can extract a subsequence (denoted in the same way for simplicity) that either \emph{vanishes,} that is,
\be{vanishing}
\limsup_{n\to\infty}\;\sup_{y\in \Real^3}\int_{y+B_R}\varphi_n^2\,dx=0\quad\forall\,R>0\,,
\ee
or satisfies the property
\be{nonvanishing}
\exists\,R_0>0\,,\;\exists\,\varepsilon_0>0\,,\;\exists\,\{y_n\}_{n\ge 1}\subset\Real^3\quad\mbox{such that}\quad\int_{y_n+B_{R_0}}\varphi_n^2\,dx\ge\varepsilon_0\,.
\ee
In the first case, for any sequence $\{y_n\}_{n\geq 1}$ in $\Real^3$, $\{\varphi_n(\cdot+y_n)\}_{n\geq 1}$ converges to zero weakly in $\Huno$. In the second case, up to the extraction of a subsequence, the sequence $\{\varphi_n(\cdot+y_n)\}_{n\geq 1}$ converges weakly towards a nonzero function $\varphi_*$ such that
\[
\int_{\Real^3} \varphi_*^2\,dx=\mu > 0\,.
\]
If $\mu=M$, then compactness (\emph{i.e.,} the strong convergence of subsequences) holds. In the opposite case, $\mu <M$, then \emph{dichotomy} occurs, that is, the splitting of the functions in at least two parts that are going away from each other: see~\cite{bi:PLL-CC1cras,bi:PLL-CC1,bi:PLL-CC2} for more details.

The concentrated--compactness method yields the strict inequalities
\be{ineqstrict}
I_M<I_{M'}+I_{M-M'}\quad\forall\,M\,,\;M'\quad\mbox{such that}\quad 0<M'<M
\ee
as necessary and sufficient conditions for the \emph{relative compactness} up to translations of all minimizing sequences. In this case, we deduce the existence of a minimizer and its orbital stability under the flow \eqref{XASP}. The proof of this equivalence is based on the fact that the only possible loss of compactness for minimizing sequences occurs either from vanishing or from dichotomy. Note that the so-called large inequalities
\be{ineqlarge}
I_M\le I_{M'}+I_{M-M'}\quad\forall\,M\,,\;M'\quad\mbox{such that}\quad 0<M'<M
\ee
always hold true due to the translation invariance. For any $\varepsilon>0$, one may indeed find $C^\infty$ functions $\phi_\varepsilon \in\Sigma_{M'}$ and $\psi_\varepsilon\in\Sigma_{M-M'}$, both with compact supports, such that $I_{M'}\le\E[\phi_\varepsilon]\le I_{M'}+\varepsilon$ and $I_{M-M'}\le\E[\psi_\varepsilon]\le I_{M-M'}+\varepsilon$. Then, for any unit vector $e$ in $\Real^3$ and for $n\in \Nat$ large enough such that $\phi_\varepsilon$ and $\psi_\varepsilon(\cdot+n\,e)$ have disjoint supports, we have $\phi_\varepsilon+\psi_\varepsilon(\cdot+n\,e)\in \Sigma_M$ and 
\[
I_M\le\limsup_{n\to+\infty} \E[\phi_\varepsilon+\psi_\varepsilon(\cdot+n\,e)]\le I_{M'}+I_{M-M'}+2\varepsilon\,.
\]
The conclusion follows since $\varepsilon$ can be made arbitrarily small. For our particular problem, it can be easily proved that vanishing cannot hold for any minimizing sequence of \eqref{minienergy} if $I_M <0$, although it might hold when $I_M=0$. This is based on Lemma I.1 in~\cite{bi:PLL-CC2} that ensures that vanishing minimizing sequences converge to zero strongly in $L^{2\alpha+2}(\Real^3)$. When $I_M=0$, vanishing has to be avoided by considering particular sequences.

Furthermore, when relative compactness up to translations can be proved for any minimizing sequence, it can also be stated that the minimizing steady state solution is orbitally stable in the sense developed in~\cite{CaLi}, thanks to the fact that mass and energy are time preserved quantities for solutions to \eqref{XASP}. In this sense, let us mention that the well-posedness of the $X^\alpha$-SP system was proved in~\cite{Caze} (Remark~6.5.3) for $\alpha\in (0,\frac23)$. For the case $\alpha=\frac23$, the existence of global solutions was proved~\cite{Caze} only for initial data with $\|\phi\|_{\Huno}$ small enough. A theory of existence of $\Ldos$ mixed-state solutions was developed in~\cite{BoLoSo} for the Slater case, $\alpha=\frac 13$. Stability properties have been proved to be false for other kind of standing waves, see for instance~\cite{BeJeLu}.

Our aim is to discuss the applicability of the concentration--compactness method to the problem \eqref{minienergy} for proving the existence of $X^\alpha$-SP \emph{steady states.} Recall that such solutions are minimizers of the energy functional under mass constraint. Let us summarize the results presented in this work in Table \ref{table:resultRef}, with some references for previously known results.
\begin{table}[ht]
\begin{tabular}{|c|c|c|c|}
\hline
$\alpha$ & Energy infimum & Existence of steady states & Ref.\\
\hline
$0$ & $I_M <0$ & No &~\cite{IlSwLa,SaSo1}\\
\hline
$(0,\frac12)$ & $I_M <0$ & Yes, for small $M$ &~\cite{CL1,SaSo,BeSiJFA,BeSiZAMP}\\
& & Open for large $M$ &\\
\hline
$\frac12$ &$I_M=0$ if $C<\frac3{\sqrt2\,\C_{1/2}}$ & No &~\cite{JeanLuo2012}\\
& $I_M=0$ if $C=\frac3{\sqrt2\,\C_{1/2}}$ & Open &\\
& $I_M <0$ if $C> \frac3{\sqrt2\,\C_{1/2}}$ & Yes &\\
\hline
$(\frac12,\frac23)$ & $I_M=0$ if $C\,M^{4\alpha-2}<V_c(\alpha)$ & No &\\
& $I_M=0$ if $C\,M^{4\alpha-2}=V_c(\alpha)$ & Yes &~\cite{JeanLuo2012}\\
& $I_M <0$ if $C\,M^{4\alpha-2}> V_c(\alpha)$ & Yes &~\cite{BeSiZAMP}\\
\hline
$\frac23$ & $I_M=0$ if $C\,M^{\frac23}\le\frac{5}{3\,\C_{2/3}}$ & No &\\
& $I_M=-\infty$ if $C\,M^{\frac23} > \frac{5}{3\,\C_{2/3}} $& No &\\
\hline
$(\frac23,2)$ & $I_M=-\infty$ & No &~\cite{BeJeLu}\\
\hline
\end{tabular}
\label{table:resultRef}
\caption{Table of existence results of steady states and related references.}
\end{table}
In this table, the constant $\C_\alpha$ denotes the optimal constant in the inequality
\[
\nrm u{2\alpha+2}^{2\alpha+2}\le\C_\alpha\,\nrm u2^{8\alpha-4}\,\D u^{2-3\alpha}\,\nrm{\nabla u}2^{6\alpha-2}\quad\forall\,u\in\Huno\,,
\]
with $\D u=4\pi\ir{u^2\,(-\Delta)^{-1}\,u^2}$. The constant
\be{Vc}
V_c(\alpha):=\frac{\alpha+1}{\C_\alpha}\left(\frac1{3\alpha-1}\right)^{3\alpha-1}\left(\frac1{2\,(2-3\alpha)} \right)^{{2-3\alpha}}
\ee
will appear in Proposition \ref{nonnegative}.

In this review, we emphasize that many partial results can been found in various papers and, concerning variational approaches,  particularly in~\cite{BeSiZAMP,BeSiJFA,JeanLuo2012,BeJeLu}. For other existence and non-existence results with the Lagrange parameter taken as a parameter, we refer to~\cite{MR1986248,DAMu,MR1896096,DR,MR2318269}. For solutions satisfying a \emph{Pohozaev constraint} (see Proposition~\ref{Prop:Phozaev}) and in particular the so-called \emph{ground state} solutions, we refer to \cite{MR2422637,DAMu,DR}.  Our contribution mostly lies in a unified framework based on the concentration-compactness method. Results corresponding to the ranges $0<\alpha<\frac 12$, $\alpha=\frac 12$ and $\frac 12<\alpha<\frac 23$ have been collected respectively in Propositions~\ref{prop:below-half}, \ref{prop:half}, and \ref{prop:above}. Our main original contribution deals with the threshold case $\alpha=\frac 12$. We also invite the reader to pay attention to the remarks of Section~\ref{Sec:steady} and to Proposition~\ref{prop:onlyVanishing} for a some open problems.\nc

In the range $\alpha\in (0,\frac12)$, we are going to prove that the strict inequalities~\eqref{ineqstrict} hold at least for $M$ small enough. The strategy of proof is inspired by~\cite{CL1} (Appendix~3) and is reproduced here for the reader's convenience. The same result has been derived in~\cite{BeSiZAMP,SaSo} for $\alpha=\frac13$ and $0<M <M_c$, and in~\cite{BeSiJFA,BeSiZAMP,CL1,SaSo} for any $\alpha\in (0,\frac12)$ and any small positive $M$. As far as the authors know, the critical case ($\alpha=\frac12$) has been treated only in~\cite{JeanLuo2012} in the specific case $C=1$, where $I_M=0$; in that case the non-existence of a minimizer has been established. We will show here that there exists a critical value for $C$, which is $3/(\sqrt2\,\C_{1/2})$, such that for larger values of~$C$ the minimizers exists but not for smaller values. The existence of minimizers for the critical value of $C$ is still an open problem, equivalent to the existence of optimal functions for the above inequality with $\alpha=\frac12$. When $\alpha\in\left(\frac12,\frac23\right)$, existence holds if and only if $M$ is large enough. The result of existence of steady states was previously obtained in~\cite{BeSiZAMP}. No steady states exist in the cases $\alpha=0$ or $\alpha\in\left[\frac23,2\right)$. The result for $\alpha=0$ is in agreement with the general dispersion property verified by the solutions to the repulsive Schr\"odinger-Poisson system proved in~\cite{IlSwLa,SaSo1}. It is also one of the motivations for introducing the local, nonlinear correction to the model. Although the existence of minimizers cannot be expected in the case $\alpha\in\left(\frac23,2\right)$ because $I_M=-\infty$, the existence and instability of other standing waves has recently been proved in~\cite{BeJeLu}. Also see \cite{MR2422637,DAMu,DR} for \emph{ground state} solutions.

For completeness, let us mention that symmetry breaking issues are not completely understood~\cite{LopesMaris,MR2926239}. In this direction, new approaches could be useful like those developed in~\cite{FelliSchneider} and subsequent papers. Stability of minimizers with null energy also raises a number of open questions.

\section{\emph{A priori} estimates and consequences}\label{sec:energy}

Before tackling the existence of steady states, we have to make sure that the minimization problem is well-posed for $\alpha\in[0,\frac23)$, and for small masses $M$ in the case $\alpha=\frac 23$. Let us first recall the Gagliardo-Nirenberg inequality
\be{GNinequality}
\nrm u{2\alpha+2}^{2\alpha+2}\le\CGN(\alpha)\,\nrm{\nabla u}2^{3\alpha}\,\nrm u2^{2-\alpha}\quad\forall\,u\in\Huno
\ee
where $\CGN(\alpha)$ is the optimal constant, depending only on $\alpha\in[0,2]$.

\begin{lemma}\label{lem:apriori} For any $\alpha\in[0,\frac 12]$, there is a positive constant $\mathrm K_\alpha$ such that, for any $u\in\Huno$, we have
\be{App:Interpolation1}
\nrm u{2\alpha+2}^{2\alpha+2}\le\mathrm K_\alpha\,\nrm u2^{2-4\alpha}\,\D u^\alpha\,\nrm{\nabla u}2^\alpha
\ee
and for any $\alpha\in[\frac 12,\frac23]$, there is a positive constant $\C_\alpha$ such that, for any $u\in\Huno$, we have
\be{App:Interpolation2}
\nrm u{2\alpha+2}^{2\alpha+2}\le\C_\alpha\,\nrm u2^{8\alpha-4}\,\D u^{2-3\alpha}\,\nrm{\nabla u}2^{6\alpha-2}\,.
\ee
\end{lemma}

The case $\alpha=\frac12$ has been established by P.-L.~Lions~\cite{Lions2} in Formula (55) page~54 and is common to the two inequalities, with $\mathrm K_{1/2}=\C_{1/2}$. The case $\alpha=\frac23$ is a special case of \eqref{GNinequality}, with $\C_{2/3}=\CGN(2/3)$. For completeness, let us give a proof.

\begin{proof} We recall that $\D u=4\pi\ir{u^2\,(-\Delta)^{-1}\,u^2}$. By expanding the square and integrating by parts, we get that
\begin{multline*}
0\le\ir{|\nabla u-a\,\nabla(-\Delta)^{-1}\,u^2|^2}\\
=\ir{|\nabla u|^2}+a^2\ir{u^2\,(-\Delta)^{-1}\,u^2}-2a\ir{u^3}\,,
\end{multline*}
that is, for an arbitrary positive parameter $a$,
\[
\ir{u^3}\le\frac1{2a}\ir{|\nabla u|^2}+\frac a2\ir{u^2\,(-\Delta)^{-1}\,u^2}\,.
\]
After optimizing on $a$, we obtain that
\be{estimlions}
\nrm u3^6\le\frac1{4\pi}\,\nrm{\nabla u}2^2\,\D u\,.
\ee
This proves \eqref{App:Interpolation1} and \eqref{App:Interpolation2} when $\alpha=\frac12$. The range $\alpha\in[0,\frac 12]$ is then covered by H\"older's inequality $\nrm u{2\alpha+2}\le\nrm u2^{2-4\alpha}\,\nrm u3^{6\alpha}$.

For $\alpha=\frac23$, \eqref{App:Interpolation2} coincides with \eqref{GNinequality}, namely
\[
\nrm u{10/3}^{10/3}\le\CGN(\tfrac23)\,\nrm{\nabla u}2^2\,\nrm u2^{4/3}\,.
\]
Hence the case $\alpha\in[\frac 12,\frac23]$ is covered by H\"older's inequality
\[
\nrm u{2\alpha+2}^{\alpha+1}\le\nrm u3^{3(2-3\alpha)}\,\nrm u{10/3}^{5(2\alpha-1)}\,.
\]
\end{proof}

Notice that from~\eqref{estimlions} we know that
\[
\C_{1/2}\le\frac1{2\,\sqrt\pi}\,.
\]

\begin{lemma}\label{boundedness} The energy functional $\mathrm E$ is bounded from below in $\Sigma_M$, if either $\alpha\in[0,\frac23)$ or $\alpha=\frac23$ and $C\,\CGN(\frac 23)\,M^{2/3}\le\frac53$. If either $\alpha\in[0,\frac23)$ or $\alpha=\frac23$ and $C\,\CGN(\frac 23)\,M^{2/3}<\frac53$, any minimizing sequence for $I_M$ is uniformly bounded in $\Huno$.\end{lemma}

\begin{proof} As a direct consequence of \eqref{GNinequality}, for every $\varphi\in \Sigma_M$ we have the estimate
\[
\E[\varphi]\ge\frac 12\,\nrm{\nabla\varphi}2^2-\frac{C\,\CGN(\alpha)}{2\alpha+2}\,M^{\frac{2-\alpha}2}\,\nrm{\nabla\varphi}2^{3\alpha}\,.
\]
\end{proof}

One of the main ingredients in our analysis is the scaling properties of the terms involved in the functional $\mathrm E$.

\begin{lemma}\label{scaling} Let $\varphi \in \Huno$. Assume that $\lambda>0$, let $p$ and $q$ be real numbers and define $\varphi_\lambda^{p,q}(x):=\lambda^p\,\varphi(\lambda^q\,x)$. Then we have
\begin{eqnarray*}
&&\ir{|\varphi_\lambda^{p,q} (x)|^2}=\lambda^{2p-3q}\,\ir{|\varphi(x)|^2}\,,\\
&&\E[ \varphi_\lambda^{p,q}]=\tfrac12\,\lambda^{2p-q}\ir{|\nabla\varphi|^2}+\tfrac14\,\lambda^{4p-5q}\,\D\varphi-\tfrac{\lambda^{(2\alpha+2)p-3q}}{2\alpha+2}\,C\ir{|\varphi|^{2\alpha+2}}\,.
\end{eqnarray*}
In the particular case $\varphi_\lambda(x):=\lambda^{\frac 32}\,\varphi(\lambda\,x)$, the mass is preserved,
\begin{multline*}
\ir{|\nabla\varphi_\lambda|^2}=\lambda^2\ir{|\nabla\varphi|^2}\,,\quad\D{\varphi_\lambda}=\lambda\,\D\varphi\,,\\
\mbox{and}\quad\ir{|\varphi_\lambda|^{2\alpha+2}}=\lambda^{3\alpha}\ir{|\varphi|^{2\alpha+2}}\,.
\end{multline*}
As a consequence, we have that $M\mapsto I_M$ is non increasing and
\[
I_M\le 0\quad\forall\,M\ge 0\,,
\]
with $I_M=-\infty$ when $\alpha>\frac 23$, for every $M>0$.
\end{lemma}

\begin{proof} The reader is invited to check the changes of variables. Let $\varphi$ be any function in $\Sigma_M$. Then, we have
\[
I_M\le\E[\varphi_\lambda]=\frac{\lambda^2}2\int_{\Real^3}|\nabla\varphi|^2\,dx+\frac{\lambda}4\,\D\varphi-\frac{\lambda^{3\alpha}\,C}{2\alpha+2}\int_{\Real^3}|\varphi|^{2\alpha+2}\,dx
\]
for all $\lambda>0$, and one concludes by letting the scaling parameter $\lambda$ go to zero that $I_M\le 0$. As a consequence of \eqref{ineqlarge}, the function $M\mapsto I_M$ is non-increasing. The last claim follows by assuming that $\alpha>\frac 23$ and by letting $\lambda$ go to infinity.\end{proof}

\begin{remark}\label{Rem:Vanishing} If $I_M=0$ for some $M>0$, we may built a minimizing sequence that converges to zero weakly in $\Huno$ by using the scaling properties. In fact, Lemma~I.1 in~\cite{bi:PLL-CC2} can be applied to any minimizing sequence in order to prove that vanishing cannot hold in the opposite case, $I_M<0$. Therefore, the condition $I_M<0$ is necessary to ensure the relative compactness up to translations of any minimizing sequence. This is the motivation for characterizing the situations in which~$\mathrm E$ reaches negative values.\end{remark}

\begin{lemma}\label{lemaequiv} Let $M>0$ and $\alpha\in[\frac13,\frac23]$. Then $\mathrm E$ takes negative values in $\Sigma_M$ if and only if the functional
\[
\varphi\mapsto\left(\frac1{3\alpha-1}\int_{\Real^3}|\nabla\varphi|^2\,dx\right)^{3\alpha-1}\left(\frac{\D\varphi}{2\,(2-3\alpha)}\right)^{2-3\alpha}-\frac C{\alpha+1}\int_{\Real^3}|\varphi|^{2\alpha+2}\,dx
\]
also takes negative values in $\Sigma_M$. Moreover, if $\alpha\in(\frac13,\frac23)$, then
\be{LowerBoundEnergy}
\E[\varphi]\ge\frac 14\,\lambda[\varphi]\,\D\varphi\left[1-\left(\frac{C\,M^{4\alpha-2}}{V_c(\alpha)}\right)^\frac1{3-2\alpha}\right]\quad\forall\,\varphi\in\Sigma_M
\ee
with $\lambda[\varphi]:=\left(\frac{3\alpha-1}{\alpha+1}\,C\,\frac{\ir{|\nabla\varphi|^2}}{\ir{\varphi^{2\alpha+2}}}\right)^\frac1{2-3\alpha}$ and $V_c(\alpha)$ given by \eqref{Vc}. \end{lemma}

Here we adopt the convention that $x^x=1$ whenever $x=0$, in order to include the endpoints of the interval. 

\begin{proof} Let $\varphi \in \Sigma_M$. Consider the family $\{\varphi_\lambda\}_{\lambda>0}$ associated with $\varphi$, such that $\nrm{\varphi_\lambda}2^2=M$ for any $\lambda>0$, as in Lemma~\ref{scaling}. We are interested in the sign of
\[
\frac 1\lambda\,\E[\varphi_\lambda]=\frac{\lambda}2\int_{\Real^3}|\nabla\varphi|^2\,dx+\frac14\,\D\varphi-\lambda^{3\alpha-1}\frac C{2\alpha+2}\int_{\Real^3}|\varphi|^{2\alpha+2}\,dx\,.
\]
In the case $\alpha=\frac13$, both potential terms in the r.h.s.~are scale invariant and we obviously have that $\mathrm E$ reaches negative values if and only if
\[
\frac14\,\D\varphi-\frac{3\,C}8 \int_{\Real^3}|\varphi|^{\frac83}\,dx <0\,.
\]
If $\alpha\in (\frac13,\frac23)$, the minimum of the r.h.s. with respect to $\lambda$ is achieved by $\lambda=\lambda[\varphi]$ and it is negative when
\[
-(2-3\alpha)\left(\frac C{2\alpha+2}\int_{\Real^3}|\varphi|^{2\alpha+2}\,dx\right)^{\frac1{2-3\alpha}}{\left(\frac1{2(3\alpha-1)}\int_{\Real^3}|\nabla\varphi|^2\,dx\right)^{\frac{1-3\alpha}{2-3\alpha}}}+\frac14\,\D\varphi<0\,.
\]
Inequality~\eqref{LowerBoundEnergy} is then a consequence of the definition of $V_c(\alpha)$. Finally, for $\alpha=\frac23$ we have that
\be{ener23}
\E[\varphi_\lambda]=\lambda^2\left(\frac12 \int_{\Real^3}|\nabla\varphi|^2\,dx-\frac{3\,C}{10} \int_{\Real^3}|\varphi|^{\frac{10}3}\,dx \right)+\lambda\,\frac14\,\D\varphi
\ee
takes negative values if and only if the leading order coefficient w.r.t.~$\lambda$,
\[
\frac12 \int_{\Real^3}|\nabla\varphi|^2\,dx-\frac{3\,C}{10} \int_{\Real^3}|\varphi|^{\frac{10}3}\,dx\,,
\]
is negative. We conclude the proof by observing that the three different conditions obtained above correspond to the precise statement of the lemma. \end{proof}

\begin{remark} In the case $\alpha=\frac23$, the functional \eqref{ener23} is not bounded from below in~$\Sigma_M$ when the leading order coefficient w.r.t.~$\lambda$ takes negative values. This remark shows the optimality of the condition on the mass stated in Lemma \ref{boundedness} for $\alpha=\frac23$. \end{remark}

In the range $\frac12<\alpha<\frac23$, we will need an additional estimate to handle the critical case corresponding to $C\,M^{4\alpha-2}=V_c(\alpha)$, that goes as follows.

\begin{corollary}\label{cor:EstimCrit} Let $\alpha\in\left(\frac12,\frac23\right)$. Then, for any $\varphi \in \Sigma_M$,
\[
\nrm\varphi{2\alpha+2}^{2\alpha+2}\le\C_{1/2}^{2-2\alpha}\,\CGN(1)^{2\alpha-1}\,M^{\alpha-\frac 12}\,\nrm{\nabla\varphi}2^{4\alpha-1}\,\D\varphi^{1-\alpha}\,.
\]
\end{corollary}

\begin{proof} Let $\varphi \in \Sigma_M$. If $\alpha\in\left(\frac12,\frac23\right)$, then we have that $3<2\alpha+2<\frac{10}3<4$. Using H\"older's inequality we get
\[
\|\varphi \|_{\L^{2\alpha+2}(\Real^3)}^{2\alpha+2}\le\|\varphi \|_{\L^3(\Real^3)}^{3(2-2\alpha)}\,\|\varphi \|_{\L^4(\Real^3)}^{4(2\alpha-1)}\,.
\]
{}From~\eqref{App:Interpolation2} written for $\alpha=\frac 12$, we know that
\[
\nrm\varphi3^3\le\C_{1/2}\,\D\varphi^\frac12\,\nrm{\nabla\varphi}2\,.
\]
On the other hand, \eqref{GNinequality} with $\alpha=1$ gives
\[
\nrm\varphi4^4\le\CGN(1)\,\nrm{\nabla\varphi}2^3\,M^\frac12
\]
Altogether, these estimates provide the result.\end{proof}

We split the analysis of the strict negativity of $I_M$ into two results, from which we will conclude that this property depends on $\alpha$ and in some cases also on the mass. Let us start with $\alpha<\frac 12$.

\begin{proposition}\label{tramo1} Let $M>0$. If $\alpha\in[0,\frac 12)$, then the functional $\mathrm E$ always reaches negative values in $\Sigma_M$. As a consequence, $I_M <0$ for all $M>0$ if $\alpha\in[0,\frac 12)$. \end{proposition}

\begin{proof} For $\alpha\in[0,\frac 13)$ the result is a trivial consequence of the mass-preserving scaling in Lemma \ref{scaling}, since we have that
\[
\lambda^{-3\alpha}\,\E[\varphi_\lambda]=\frac12\,{\lambda^{2-3\alpha}}\ir{|\nabla\varphi|^2}+\frac14\,{\lambda^{1-3\alpha}}\,\D\varphi-\frac C{2\alpha+2} \int_{\Real^3} |\varphi|^{2\alpha+2}\,dx
\]
is negative for any non-trivial $\varphi\in\Huno$ if $\lambda>0$ is chosen small enough.

To complete the proof for $\alpha\in[\frac13,\frac12)$, it remains to find a particular test function $\varphi\in\Sigma_M$ with negative energy for any $M>0$. We follow a classical approach in the literature on the concentration--compactness method, see for instance~\cite{Lions1992}. Consider $M>0$ and $\eta \in \Sigma_M$ such that ${\hbox{\rm{supp}}}(\eta) \subset B(0,1)$, where $B(0,1)$ denotes the unit sphere centered at $0$. For any positive integer $n$, define $\eta_n(x):=\eta(n^\frac13 x)$. Then the support of $\eta_n$ is contained in $B(0,1)$ and by direct calculations we have
\begin{eqnarray*}
&\|\eta_n\|_{\Ldos}^2=\frac1n\,\|\eta\|_{\Ldos}^2\,,\quad\D{\eta_n}=\frac1{n^{5/3}}\,\D\eta\,,&\\
&\|\eta_n\|_{\L^{2\alpha+2}(\Real^3)}^{2\alpha+2}=\frac1n\,\|\eta\|_{\L^{2\alpha+2}(\Real^3)}^{2\alpha+2}\,,\quad\int_{\Real^3}|\nabla\eta_n|^2\,dx=\frac1{n^{1/3}}\int_{\Real^3}|\nabla\eta|^2\,dx\,.&
\end{eqnarray*}
Let $n$ be a given integer bigger than $1$ and let us consider the test function $\varphi(x):=\sum_{i=1}^n \eta_n(x-x_i)$, where the points $x_i\in\Real^3$, $i=1,\dots n$ are chosen such that
\[
|x_i-x_j|\ge\frac{{M^2}}{\D\eta}\,n^{2/3}+2\quad\forall\,i\neq j\,.
\]
By definition $\varphi$ verifies $\|\varphi\|_{\Ldos}^2=\|\eta\|_{\Ldos}^2=M$, $\|\varphi\|_{\L^{2\alpha+2}(\Real^3)}^{2\alpha+2}=\| \eta\|_{\L^{2\alpha+2}(\Real^3)}^{2\alpha+2}$ and $\int_{\Real^3}|\nabla\varphi|^2\,dx=n^{2/3}\int_{\Real^3}|\nabla\eta|^2\,dx$. Now, we estimate $\D{\varphi}$ as follows:
\begin{eqnarray*}
\D{\varphi} &=&\sum_{i,\,j=1}^n\iint_{\Real^3\times\Real^3}{|\eta_n(x-x_i)|^2\,|\eta_n(x'-x_j)|^2}\frac{dx\,dx'}{|x-x'|}\\
&=&n\,\D{\eta_n}+\sum_{j \neq i} \iint_{\Real^3\times\Real^3}\frac{|\eta_n(x)|^2\,|\eta_n(x')|^2}{|x+x_i-x'-x_j|}\,dx\,dx'\\
&\le&\frac{\D\eta}{n^{2/3}}+\sum_{j\neq i}\iint_{\Real^3\times\Real^3}\frac{|\eta_n(x)|^2\,|\eta_n(x')|^2}{|x_i-x_j|-2}\,dx\,dx'\\
&\le &\frac{\D\eta}{n^{2/3}}+\frac{\D\eta}{{M^2}\,n^{2/3}}\,\frac{{M^2\,n(n-1)}}{2\,n^2}=\frac{2\,\D\eta}{n^{2/3}}\,.
\end{eqnarray*}
Combining these estimates and Lemma \ref{lemaequiv} with the fact that $(3\alpha-1)-(2-3\alpha)<0$ if $\alpha<\frac12$, we are done with the proof. \end{proof}

If $\alpha\in (\frac12,\frac23]$, the functional $\mathrm E$ might not reach negative values depending on the value of the mass $M$ and the constant~$C$, as stated in the following result.

\begin{proposition}\label{nonnegative} In the case $\alpha\in[\frac12,\frac23]$, $I_M=0$ if and only if
\be{condposi}
C\,M^{{4\alpha-2}}\le V_c(\alpha)
\ee
holds, where the constant $V_c(\alpha)$ is given in \eqref{Vc}. On the contrary, if \eqref{condposi} does not hold, then $I_M$ is negative.\end{proposition}

We recall that $V_c(\alpha)=\frac{\alpha+1}{\C_\alpha}\left(\frac1{3\alpha-1}\right)^{3\alpha-1}\left(\frac1{2\,(2-3\alpha)} \right)^{{2-3\alpha}}$ where $\C_\alpha$ is the optimal constant in \eqref{App:Interpolation2}.

\begin{proof} According to Lemma \ref{lemaequiv}, $I_M=0$ for $\alpha\in[\frac12,\frac23]$ if and only if
\[
\nrm\varphi{2\alpha+2}^{2\alpha+2}\le\frac{\alpha+1}C\left(\frac{\nrm{\nabla\varphi}2^2}{3\alpha-1}\right)^{3\alpha-1}\left(\frac{\D\varphi}{2\?(2-3\alpha)}\right)^{2-3\alpha}
\]
for all $\varphi \in \Sigma_M$. Comparing with the definition of $\C_\alpha$ in \eqref{App:Interpolation2}, this clearly entails that $I_M=0$ if and only if \eqref{condposi} holds. According to Lemma~\ref{scaling}, $I_M$ is negative (and eventually $-\infty$) otherwise.\end{proof}

Although our problem is originally set in the framework of complex valued functions, we finally observe that we can reduce it to non-negative real valued functions.

\begin{lemma} Consider a complex valued minimizer $\psi$ to the problem \eqref{minienergy}. Then, the real function $|\psi|$ is also a minimizer for \eqref{minienergy}.\end{lemma}

\begin{proof} It is well known that if $\psi \in \Sigma_M$, then $|\psi|$ also belongs to $\Sigma_M$. Since the potential energy only depends on $|\psi|^2$, it takes the same value on $\psi$ and $|\psi|$. On the other hand, the kinetic enegy verifies
\[
\int_{\Real^3} \big|\nabla|\psi|\,\big|^2\,dx\le\int_{\Real^3}\Big(\nabla|{\hbox{\rm{Re}}}\,\psi|^2+\nabla|{\hbox{\rm{Im}}}\,\psi|^2\Big)\,dx=\int_{\Real^3}|\nabla\psi|^2\,dx
\]
as a consequence of the convexity inequality for gradients~\cite{LiebLoss}, where equality holds if and only if $|{\hbox{\rm{Re}}}\,\psi(x)|=c\,|{\hbox{\rm{Im}}}\,\psi(x)|$ for some constant $c$. Hence, $|\psi|$ is also a minimizer. \end{proof}

If $I_M$ is achieved, we can then prove the \emph{Virial Theorem} relation for the terms of the energy functional by using their scaling properties.

\begin{proposition}\label{Prop:Phozaev} Assume that $0<\alpha<\frac23$. Any minimizer $\varphi_M$ of $I_M$ satisfies
\be{eqnvirial}
\int_{\Real^3}|\nabla\varphi_M|^2\,dx+\frac14\,\D{\varphi_M}-\frac{3\,\alpha\,C }{2\alpha+2} \int_{\Real^3}|\varphi_M|^{2\alpha+2}\,dx=0\,.
\ee
\end{proposition}

\begin{proof} Let us assume that there exists a minimizer $\varphi_M\in\Sigma_M$ of $I_M$. According to Lemma~\ref{scaling}, for every $\lambda>0$ the rescaled function $\varphi_{M,\lambda}=\lambda^{3/2}\,\varphi_M(\lambda\,\cdot)$ also lies in~$\Sigma_M$. The function $\lambda\mapsto \E[\varphi_{M,\lambda}]$ attains its minimal value at $\lambda=1$. Since
\[
\E[\varphi_{M,\lambda}]=\frac12\,\lambda^2\int_{\Real^3}|\nabla\varphi_M|^2\,dx+\lambda\,\frac14\,\D{\varphi_M}-\lambda^{3\alpha}\,\frac C{2\alpha+2}\int_{\Real^3}|\varphi_M|^{2\alpha+2}\,dx\,,
\]
the cancellation of the derivative with respect to $\lambda$ at $\lambda=1$ provides with \eqref{eqnvirial}. \end{proof}

At this stage, we can write down the Euler-Lagrange equation corresponding to the minimization problem $I_M$ and deduce an energy identity.

\begin{lemma}\label{Lem:3.1} Assume that $0<\alpha<\frac23$. Any minimizer $\varphi_M$ of $I_M$ satisfies~\eqref{eq:sw} and 
\be{intEL}
\int_{\Real^3}|\nabla\varphi_M|^2\,dx+\D{\varphi_M}-C\int_{\Real^3}|\varphi_M|^{2\alpha+2}\,dx+\ell_M\,M=0\,.
\ee
In particular, at least for $\alpha \in (0,\frac15] \cup (\frac12,\frac23)$, we have $\ell_M>0$. If $\alpha=\frac12$, then $\ell_M=-\frac6M\,I_M\geq 0$. 

\end{lemma}

\begin{proof} Identity~\eqref{intEL} is obtained by multiplying the Euler-Lagrange equation~\eqref{eq:sw} by $\varphi_M$ and integrating by parts. If we eliminate $\D{\varphi_M}$ and $\nrm{\varphi_M}{2\alpha+2}$ from~\eqref{eqnvirial}, \eqref{intEL} and use
\[
E[\varphi_M]=\frac 12\int_{\Real^3}|\nabla\varphi_M|^2\,dx+\frac 14\D{\varphi_M}-\frac C{2\alpha+2}\int_{\Real^3}|\varphi_M|^{2\alpha+2}\,dx=-|I_M|\,,
\]
we complete the proof using 
\[
\ell_M=\frac 2M\left(\frac{2\alpha-1}{3\alpha-1}\int_{\Real^3}|\nabla\varphi_M|^2\,dx+\frac{5\alpha-1}{3\alpha-1}\,|I_M|\right)\,.
\]
\end{proof}

\begin{corollary}\label{Cor:CriticalMass} Assume that $\alpha\in(0,\frac 12)\cup(\frac 12,\frac23)$. Any minimizer $\varphi_M$ of $I_M$ is such that
\begin{eqnarray*}
&&\ir{|\nabla\varphi_M|^2}=\frac 12\,(3\alpha-1)\,\varepsilon_M-(5\alpha-1)\,\eta_M\\
&&\D{\varphi_M}=(2-3\alpha)\,\varepsilon_M-2\,(2-\alpha)\,\eta_M\\
&&\ir{|\varphi_M|^{2\alpha+2}}=\frac 14\,\varepsilon_M-\frac 32\,\eta_M
\end{eqnarray*}
where
\[
\varepsilon_M:= \frac{M\,\ell_M}{2\alpha-1}\quad\mbox{and}\quad\eta_M:=\frac{I_M}{1-2\alpha}\,.
\]
\end{corollary}

\begin{proof} The proof is a straightforward consequence of $\E[\varphi_M]=I_M$, \eqref{eqnvirial} and~\eqref{intEL}.\\\end{proof}

Lemma~\ref{Lem:3.1} has interesting consequences concerning the decay of the minimizers, that can be derived from Lemma 19 and Theorem 6 in~\cite{BoMe2011}, as shown in the following result. Also see Theorem~1.3 in~\cite{BeJeLu} and Theorem~6.1 in~\cite{doi:10.1142/S0219199712500034} for related results.

\begin{lemma} Consider a nonnegative solution to \eqref{eq:sw} such that
\[
\frac12\int_{\Real^3}|\nabla\varphi_M|^2\,dx+\,\frac14\,\D{\varphi_M}+\ell_M\,\ir{\varphi_M^2}<\infty\,,
\mbox{ with }\ell_M\ge 0\,.
\] 
Then, there exist positive constants $K$ and $\delta$ such that
\[
\varphi_M(x)\le K\,e^{-\delta\,\sqrt{1+|x|}}\quad\forall\,x\in\Real^3\,.
\]
\end{lemma}

In the case $\ell_M=0$, this result ensures that the above solution belongs to $\Huno$, since the exponential decay also guarantees that the minimizer is in $\L^2(\Real^3)$.

\medskip\noindent {\bf The rescaled problem.} Given that our main tool in proving the existence of minimizers will consist in checking the strict inequalities \eqref{ineqstrict}, we are going to study the infimum value $I_M$ as a function of the mass $M$. To this purpose, we fix a function $\varphi_1\in \Sigma_1$ and apply the scaling properties in Lemma~\ref{scaling} with $2p-3q=1$ and $\lambda=M$. We denote by $\varphi_{M,p}$ be the corresponding rescaled function. Then, according to Lemma \ref{scaling} we have that $\varphi_{M,p}\in \Sigma_M$ and
\be{eq:enerscal2}
\E[\varphi_{M,p}]=\tfrac12\,M^{\frac{4p+1}3}\,\nrm{\nabla\varphi_1}2^2+\tfrac14\,M^\frac{2p+5}3\,\D{\varphi_1}-\tfrac C{2\alpha+2}\,M^{2\alpha p+1}\,\nrm{\varphi_1}{2\alpha+2}^{2\alpha+2}\,,
\ee
for any real number $p$.

\section{Existence and non-existence of steady states}\label{Sec:steady}

In this section we analyze the existence of minimizers for the variational problem~\eqref{minienergy}.
\subsection{Non-existence results when \texorpdfstring{$\alpha=0$ or $\alpha=2/3$}{alpha=0 or alpha=2/3}}\label{ssec:nonexistence}

In the case $\alpha=0$, the minimization problem reduces to
\begin{eqnarray*}
I_M&=&\inf\left\{\frac12 \int_{\Real^3}|\nabla\varphi|^2\,dx
+\frac14\,\D\varphi-\frac C2\int_{\Real^3}|\varphi|^2\,dx\,:\,\varphi\in\Sigma_M\right\}\\
&=& \inf\left\{\frac12 \int_{\Real^3}|\nabla\varphi|^2\,dx
+\frac14\,\D\varphi\,:\,\varphi\in\Sigma_M\right\}-\frac C2\,M=-\frac C2\,M\,,
\end{eqnarray*}
by a scaling argument. Therefore, $I_M$ is never achieved when $M>0$ (despite it is always negative) since any possible minimizer would make the gradient term vanish, and then should vanish itself in $\Real^3$.

In the case $\alpha=\frac23$, either $I_M=0$ or $I_M=-\infty$, and in both cases there are no minimizers. Actually, $I_M=0$ if and only if
\[
\frac12 \int_{\Real^3}|\nabla\varphi|^2\,dx-\frac{3\,C}{10}\int_{\Real^3}|\varphi|^{\frac{10}3}\,dx\ge0\,.
\]
See Lemma \ref{lemaequiv} and its proof for details. Hence, the minimum cannot be attained, otherwise $\D{\varphi}=0$, which is absurd. 
\vskip6pt
{}From now on we shall assume that $0<\alpha<\frac23$. We first examine the range $0<\alpha<\frac 12$ in Subsection~\ref{ssec:below}. Subsection~\ref{ssec:half} is devoted to the special limiting case $\alpha=\frac12$. Finally, the range $\frac12<\alpha<\frac 23$ is analyzed in Subsection~\ref{ssec:above}.

\subsection{The interval \texorpdfstring{$0<\alpha<\frac 12$}{0<alpha<1/2}}\label{ssec:below}
We prove the following~:

\begin{proposition}\label{prop:below-half} Let $0<\alpha<\frac 12$. Then, for $M>0$ small enough, the strict inequalities 
\[
I_M <I_{M'}+I_{M-M'}
\]
hold for every $M'$ such that $0<M'<M$. In particular, all minimizing sequences are compact in $\Huno$ up to translations and the extraction of a subsequence. Therefore, $I_M$ is attained for $M$ small enough. 
\end{proposition}

\begin{proof} In Proposition~\ref{tramo1} we have proved that
$I_M<0$ for every $M>0$. In the range $\alpha\in(0,\frac12)$, we may choose the parameter $p$ in the rescaled problem \eqref{eq:enerscal2} such that
\begin{equation*}
0\le\,\frac{4p+1}3= 2\alpha p+1\,<\frac{2p+5}3\,,
\end{equation*}
\emph{i.e.,} such that the gradient and the power term are of the same order for small $M$ and dominate the Poisson energy in this regime. With this choice we can deduce
\be{eq:defJ}
I_M=M^{\frac{2-\alpha}{2-3\alpha}}\,J_1^M
\ee
where, for every $\mu>0$, 
\[
J_\mu^M=\inf\left\{\frac12\,\nrm{\nabla\varphi}2^2-\frac C{2\alpha+2}\,\nrm\varphi{2\alpha+2}^{2\alpha+2}+M^{\frac{2(1-2\alpha)}{2-3\alpha}}\,\frac14\,\D\varphi\,:\,\varphi\in\Sigma_\mu\right\}\,.
\]
Note that the same scaling argument shows that
\be{eq:scalJ}
J_\mu^M=\mu^{\frac{2-\alpha}{2-3\alpha}}\,J_1^{\mu M}\,.
\ee
With $\mu=\frac{M'}M$ and using \eqref{eq:defJ} and \eqref{eq:scalJ}, it is easily proved that the strict inequalities of Proposition~\ref{prop:below-half} are equivalent to 
\be{ineqstrictJ}
J_1^M<J_\mu^M+J_{1-\mu}^M\quad\forall\,\mu\in(0,1)\,.
\ee 
We are going to prove that the above strict inequalities hold for $M$ small enough.

Observe now that $M^{\frac{2(1-2\alpha)}{2-3\alpha}}$ goes to zero as $M$ does, and $\lim_{M\rightarrow 0}J_1^M=J_1^0$. The key point is that for every $\lambda>0$, $J_\lambda^0$ satisfies the strict inequalities of the concentration--compactness principle, namely 
 \be{CCJ}
J_\lambda^0<J_\mu^0+J_{\lambda-\mu}^0\,,\quad\forall\,\mu\in(0,\lambda)\,.
\ee
This is an immediate consequence of the fact that $J_\lambda^0=\lambda^{\frac{2-\alpha}{2-3\alpha}}\,J_1^0$, with 
$J_1^0<0$ and $\frac{2-\alpha}{2-3\alpha}>1$. The sign of $J_1^0$ is deduced from the scaling argument in Lemma~\ref{scaling}, by observing that the negative term dominates the gradient contributions for $3\alpha<2$. We now
prove that $J^M_1$ satisfies the strict inequalities \eqref{ineqstrictJ} for $M$ small enough. We argue by contradiction
assuming that this is not the case. Then, there exist a sequence $\{M_n\}_{n\ge 1}$ going to $0$ and a sequence $\{\lambda_n\}_{n\ge 1}$ in $(0,1)$ such that
\be{pasCC}
J_1^{M_n}=J_{\lambda_n}^{M_n}+J_{1-\lambda_n}^{M_n}\,.
\ee
Assume that $\frac 12\le\lambda_n<1$ (if $\lambda_n\in(0,\frac 12)$, we may exchange the roles of $\lambda_n$ and $1-\lambda_n$). By continuity with respect to $M$ we conclude that $\lambda_n \to 1$, otherwise we get a contradiction with~\eqref{CCJ}. In addition, we may choose as $\lambda_n$ the infimum of the set $\{\lambda\in[\frac 12,1)\,:\,J_1^{M_n}=J_{\lambda}^{M_n}+J_{1-\lambda}^{M_n}\}$.

We now claim that, for $n$ large enough, $J_{\lambda_n}^{M_n}$ satisfies the strict inequalities of the concentration--compactness principle
\be{CCJn}
J_{\lambda_n}^{M_n}<J_{\mu}^{M_n}+J_{\lambda_n-\mu}^{M_n}\quad\forall\,\mu\in(0,\lambda_n)\,.
\ee
If not, there exists a sequence $\{\mu_n\}_{n\ge1}$ with $\mu_n\in(\frac12\,\lambda_n,\lambda_n)$ such that
\be{contra}
J_{\lambda_n}^{M_n}=J_{\mu_n}^{M_n}+J_{\lambda_n-\mu_n}^{M_n}\,.
\ee
Then, from \eqref{pasCC} and \eqref{contra} we find
\begin{eqnarray*}
J_{\mu_n}^{M_n}+J_{1-\mu_n}^{M_n}&\ge &J_1^{M_n}=J_{\mu_n}^{M_n}+J_{\lambda_n-\mu_n}^{M_n}+J_{1-\lambda_n}^{M_n}\\
&\ge&J_{\mu_n}^{M_n}+J_{1-\mu_n}^{M_n}\,,
\end{eqnarray*}
for the reverse large inequalities $J_{1-\mu_n}^{M_n}\le J_{1-\lambda_n}^{M_n}+J_{1-\mu_n-(1-\lambda_n)}^{M_n}$ always hold true. Hence, the equality 
\begin{equation}\label{contra2}
J_{\mu_n}^{M_n}+J_{1-\mu_n}^{M_n}=J_1^{M_n}
\end{equation}
is verified. By definition of $\lambda_n$ and since $\mu_n\geq \frac{\lambda_n}2$ with $\lambda_n\geq \frac12$, we must have $\frac 14\leq \mu_n<\frac 12$. Extracting a subsequence if necessary, we may assume that $\mu_n$ converges to $\mu$ with $\frac 14\leq \mu\leq\frac 12$. Passing to the limit in \eqref{contra2}, we get $J_{\mu}^{0}+J_{1-\mu}^{0}=J_1^{0}$ with $\mu\in (0,1)$ thereby reaching a contradiction with \eqref{CCJ}. So far we have proved that the strict inequalities \eqref{CCJn} hold. 

In particular, for $n$ large enough, there exists a minimizer $\varphi_n$ of $J_{\lambda_n}^{M_n}$ such that $\{(\lambda_n)^{-1/2}\varphi_n\}_{n\ge 1}$ is a minimizing sequence for $J_1^0$. Since \eqref{CCJ} holds, this sequence converges strongly in $\Huno$ up to translations to a minimizer $\varphi_\infty$ of $J_1^0$. The same holds for $\{\varphi_n\}_{n\ge 1}$, given that $\lambda_n\to 1$. Without loss of generality we may assume that $\varphi_n$ and $\varphi_\infty>0$ satisfy the respective Euler--Lagrange equations in $\Real^3$
\[
-\Delta\varphi_n-C\,\varphi_n{}^{2\alpha+1}+M_n^{\frac{2(1-2\alpha)}{2-3\alpha}}\,\Big(\varphi_n^2\ast\frac 1{|x|}\Big)\varphi_n+\theta_n\,\varphi_n=0\,,
\]
with $\nrm{\varphi_n}2^2=\lambda_n$ and
\[
-\Delta\varphi_\infty-C\,\varphi_\infty^{2\alpha+1}+\theta_1\,\varphi_\infty=0\,,
\]
with $\nrm{\varphi_\infty}2^2=1$ and $\theta_1>0$. Having in mind to contradict \eqref{pasCC} we argue as follows. We first write
\[
\frac{J_1^{M_n}-J_{\lambda_n}^{M_n}}{1-\lambda_n}=\frac{J_{1-\lambda_n}^{M_n}}{1-\lambda_n}\,.
\]
As $\lambda_n$ goes to $1$, the left-hand side can be bounded from above by $-\,\theta_1$, while from~\eqref{eq:scalJ} the quotient
\begin{equation*}
\frac{J_{1-\lambda_n}^{M_n}}{1-\lambda_n}=(1-\lambda_n)^{\frac{2\alpha}{2-3\alpha}}\,J_1^{(1-\lambda_n)M_n}
\end{equation*}
converges to $0$ because $J_1^{(1-\lambda_n)M_n}$ converges to $J_1^0$ as $\lambda_n\to 1$, and $\frac{2\alpha}{2-3\alpha}$ is positive.\\
\end{proof}

\begin{remark} The general case for any $M$ is still an open problem. The possibility of dichotomy is the delicate case to be analyzed since vanishing is easily ruled out by the fact that $I_M$ is negative. \end{remark}

\subsection{The limiting case \texorpdfstring{$\alpha=\frac 12$}{alpha1/2}}\label{ssec:half}

Our main result is the following.

\begin{proposition}\label{prop:half} Let $\C_{1/2}$ be the best constant in \eqref{App:Interpolation2} with $\alpha=\frac12$.
\begin{itemize}
\item[(i)] If $\frac3{\sqrt2\,\C_{1/2}}> C$, then $I_M=0$ and $I_M$ is not achieved for any $M>0$.
\item[(ii)] If $\frac3{\sqrt2\,\C_{1/2}}<C$, then $I_M<0$ and $I_M$ is achieved for every $M>0$. In addition, all
minimizing sequences are relatively compact in $\Huno$ up to translations. \end{itemize}
\end{proposition}

\begin{remark} The remaining case $\frac3{\sqrt2\,\C_{1/2}}=C$, where $I_1=0$, might be attained if and only if P.-L.~Lions' inequality \eqref{estimlions} has an optimal function in $\Sigma_1$. This is, for the moment, an open question. \end{remark}

\begin{proof} As an immediate consequence of the scaling formulae of Lemma~\ref{scaling}, by taking $p=2$ in \eqref{eq:enerscal2}, we have that 
\be{eq:rel-half}
I_M=M^3\,I_1\,,
\ee
for every $M>0$, and $I_M$ is achieved if and only if $I_1$ is also achieved. This is the only case in which all powers of $M$ appearing in the right-hand side of \eqref{eq:enerscal2} are identical. When $I_1<0$, it is a well-known fact~\cite{bi:PLL-CC1cras,bi:PLL-CC1} that the relation \eqref{eq:rel-half} implies the strict inequalities \eqref{ineqstrict}, hence the result. Indeed, the strict inequalities \eqref{ineqstrict} hold as a consequence of the convexity of $M\mapsto M^3$.

Assume now that $I_1=0$, so that $I_M=0$ for every $M>0$. We assume that $I_1$ is achieved by some function $\varphi_1$ in $\Huno$. Then $\varphi_1$ satisfies the Euler--Lagrange equation \eqref{eq:sw} with a zero Lagrange multiplier (since it is also a minimizer without any constraint on the $\L^2(\Real^3)$ norm). Also see~Lemma~\ref{Lem:3.1} for a direct proof. If we apply the corresponding equation to $\varphi_1$, integrate over $\Real^3$ and use the information $I_1=\E[\varphi_1]=0$, we deduce
\begin{equation*}
\frac12\int_{\Real^3}|\nabla\varphi_1|^2\,dx=\frac14\,\D{\varphi_1}=\frac C6\int_{\Real^3}|\varphi_1|^3\,dx\,.
\end{equation*}
Hence, by definition of $\C_{1/2}$ we obtain
\[
\frac1{\C_{1/2}}\le\frac{C\,\sqrt2}3\,,
\]
or equivalently
\[
\frac3{\sqrt2\,\C_{1/2}}\le C\,.
\]
Therefore, using \eqref{condposi}, the equality $I_1=0$ can be achieved only when
\[
\frac3{\sqrt2\,\C_{1/2}}=C\,.
\]
As a consequence, $I_1$ (and, up to a scaling, $I_M$) is attained if and only if the optimal constant in~\eqref{estimlions} is attained {by a minimizer in $\L^2(\Real^3)$}.
\end{proof}

\medskip We conclude this section by examining the critical case $\alpha=\frac 12$ with the limiting constant $C=\frac3{\sqrt2\,\C_{1/2}}$. The problem is open, but we can prove that lack of compactness may occur only by vanishing as shown by the following result.
\begin{proposition}\label{prop:onlyVanishing} Assume that $\alpha=\frac12$ and $C=\frac3{\sqrt2\,\C_{1/2}}$. Let $\{\phi_n\}_{n\ge 1}$ be a minimizing sequence for $I_1$. If \emph{vanishing does not occur,} that is, if~\eqref{nonvanishing} holds,
then there exists a minimizer for $I_1$.
\end{proposition}

By Lemma \ref{boundedness}, $\{\phi_n\}_{n\ge 1}$ is bounded in $\Huno$. Since $I_1$ is invariant by translation, relative compactness in $\Huno$ may only be expected up to translations. Also, since $I_\lambda=0$ for every $\lambda>0$, concentration--compactness type inequalities turn into equalities. In particular, there exist minimizing sequences that are not relatively compact in $\Huno$, up to any translations. According to the concentration--compactness terminology~\cite{bi:PLL-CC1cras,bi:PLL-CC1,bi:PLL-CC2}, either $\{\phi_n\}_{n\ge 1}$ fulfills \eqref{vanishing} and \emph{vanishing} occurs, or \eqref{nonvanishing} holds. If there exists some minimizing sequence for which vanishing does not occur, we will now prove that existence of a minimizer is guaranteed. 

\begin{proof} We first show that \eqref{nonvanishing} ensures the existence of a minimizer. Indeed, the new minimizing sequence $\{\phi_n(\cdot+y_n)\}_{n\ge 1}$ converges (up to a subsequence) to a function $\phi$ in $\Huno$, weakly in $\Huno$ and in $\L^p(\Real^3)$ for every $2\le p\le 6$, strongly in $\L^p_{\rm loc}(\Real^3)$ for every $1\le p<6$ (by the Rellich-Kondrachov theorem); consequently, it also converges almost everywhere in $\Real^3$. The condition \eqref{nonvanishing} guarantees that $\phi\neq 0$ since $\int_{B_{R_0}}\phi^2\,dx\ge\varepsilon_0$ by passing to the limit as $n$ goes to infinity. Let $\mu=\int_{\Real^3}\phi^2\,dx$ with $0<\mu\le 1$.

If $\mu=1$, we are done~: $\{\phi_n(\cdot+y_n)\}_{n\ge 1}$ converges to $\phi$ strongly in $\L^2(\Real^3)$, and therefore in $\L^p(\Real^3)$ for every $2\le p<6$ by H\"{o}lder's inequality. In particular, the convergence is also strong in $\L^3(\Real^3)$ and $0=\liminf_{n\to+\infty}\E[\phi_n]\ge\E[\phi)]\geq I_1$. Hence, $\E[\phi]=0$ and $\phi$ is a minimizer of $I_1$. In addition, the convergence is strong in $\Huno$ since all above inequalities turn into equalities.

If $\mu<1$, we are in the so-called dichotomy case. We shall prove that $\phi$ is a minimizer of $I_\mu$. Then, according to Lemma~\ref{prop:half}, $I_1$ is also achieved. Let us define $r_n:=\phi_n(\cdot+y_n)-\phi$. Then, $\{r_n\}_{n\ge 1}$ is bounded in $\Huno$. Up to a subsequence, it converges to $0$ weakly in $\Huno$ and in $\L^p(\Real^3)$ for every $2\le p<6$, strongly in $\L^p_{\rm loc}(\Real^3)$ for every $1\le p<6$, and almost everywhere in $\Real^3$. In addition, by taking weak limits we find
\begin{eqnarray}\label{eq:dic-grad}
&&\lim_{n\to+\infty}\int_{\Real^3} r_n^2\,dx=1-\mu\,, \nonumber \\ &&\int_{\Real^3} \vert \nabla\phi_n\vert ^2\,dx=\int_{\Real^3} \vert \nabla r_n\vert ^2\,dx+\int_{\Real^3} \vert \nabla\phi \vert ^2\,dx+o_n(1)\,,
\end{eqnarray}
where $o_n(1)$ is a shorthand for a quantity that goes to $0$ when $n$ goes to infinity. Using Theorem~1 in~\cite{bi:BrL1}, we have
\be{eq:dic-power}
\int_{\Real^3}\vert \phi_n\vert^3\,dx=\int_{\Real^3}\vert r_n\vert^3\,dx+\int_{\Real^3}\vert \phi\vert^3\,dx+o_n(1)\,.
\ee
We first check as in~\cite{CL1} that
\be{eq:conv-0}
\lim_{n\to+\infty} \Vert \phi\,r_n\Vert_{\L^p(\Real^3)}=0\quad\forall\,p\in[1,3)\,.
\ee
We just argue for $p=1$, as the analysis for the other powers follows by interpolation. Since $\{r_n\}_{n\ge 1}$ converges strongly to $0$ in $\L^2_{\rm loc}(\Real^3)$, $\{\phi\,r_n\}_{n\ge 1}$ converges strongly to~$0$ in $\L^1_{\rm loc}(\Real^3)$ as $n \to \infty$. Next, for every $R>0$ we have
\[
\int_{|x|\ge R} \vert \phi\,r_n\vert\,dx\le\Big(\int_{|x|\ge R} \vert \phi\vert^2\,dx\Big)^{1/2}\,\Big(\int_{|x|\ge R} \vert r_n\vert^2\,dx\Big)^{1/2}.\]
The first term in the right-hand side may be taken arbitrarily small for $R$ large enough since $\phi \in \Ldos$, while the second one is bounded independently of $n$ and~$R$ since $\{r_n\}_{n\ge 1}$ is bounded in $\Ldos$. Writing $\int_{\Real^3} \vert \phi\,r_n\vert\,dx=\int_{|x|\le R} \vert \phi\,r_n\vert\,dx+\int_{|x|\ge R} \vert \phi\,r_n\vert\,dx$ we get the result. By writing down
\[
\int_{\Real^3}\vert \phi_n\vert^3\,dx-\int_{\Real^3}\vert r_n\vert^3\,dx-\int_{\Real^3}\vert \phi\vert^3\,dx=3\int_{\Real^3}\vert \phi\,r_n\vert\,(\vert r_n\vert+\vert \phi\vert)\,dx\,,
\]
we obtain \eqref{eq:dic-power} since $\{\vert r_n\vert+\vert \phi\vert\}_{n\ge 1}$ is bounded in $\Ldos$ and $\{\phi\,r_n\}_{n\ge 1}$ converges to $0$ in $\L^2(\Real^3)$. Finally, we check that
\be{eq:dic-D}
\liminf_{n\to +\infty}\D{\phi_n}\geq \D\phi+\liminf_{n\to +\infty}\D{r_n}\,.
\ee
On the one hand, since $\{\phi_n\}_{n\ge 1}$ is bounded in $\Huno$, $\{\phi_n^2\star \frac 1{|x|}\}_{n\ge 1}$ is bounded in $\L^\infty(\Real^3)$ thanks to
\begin{eqnarray*}
\Big\|\phi^2\star\frac 1{|x|}\Big\|_{\L^\infty(\Real^3)}=\sup_{x\in \Real^3}\int_{\Real^3}
\frac{\phi^2(y)}{|x-y|}\,dy&\le&\sup_{x\in \Real^3}\Big(\int_{\Real^3} \frac{\phi^2(y)}{|x-y|^2}\,dy\Big)^{1/2}\,\nrm\phi2\\
&\le& 2\,\|\nabla\phi\|_{\Ldos}\,\|\phi\|_{\Ldos}\,,
\end{eqnarray*}
where we have used Cauchy-Schwarz' inequality and Hardy's inequality.
Then, we have
\[
\Big\vert \int_{\Real^3}\Big(\phi_n\star\frac 1{|x|}\Big)\,(\phi\,r_n)\,dx\Big\vert\le\Big\Vert \phi_n^2\star \frac1{|x|}\Big\Vert_{\L^\infty(\Real^3)}\,\Vert \phi\,r_n\Vert_{\L^1(\Real^3)}\,,
\]
and hence
\[
\lim_{n\to\infty}\int_{\Real^3}\Big(\phi_n\star\frac 1{|x|}\Big)\,(\phi\,r_n)\,dx=0\,,
\]
because of \eqref{eq:conv-0}. On the other hand, $\int_{\Real^3}\big((\phi\,r_n)\star\frac 1{|x|}\big)\,(\phi\,r_n)\,dx \geq 0$. Actually it is also converging to $0$ as $n\to\infty$, and \eqref{eq:dic-D} follows. Gathering together \eqref{eq:dic-grad}, \eqref{eq:dic-power} and \eqref{eq:dic-D}, we obtain
\begin{eqnarray*}
0=\limsup_{n\to+\infty}\E[\phi_n]=\liminf_{n\to+\infty}\E[\phi_n]&=&\E[\phi]+\liminf_{n\to+\infty}\E[r_n]\\
&&\quad\ge\E[\phi]\ge I_\mu=0
\end{eqnarray*}
since $\E[r_n]$ is nonnegative for every $n\ge1$. Then, all above inequalities turn into equalities. In particular, $I_\mu=\E[\phi]=0$ is attained. By Lemma~\ref{prop:half}, $I_1$ also is attained. This concludes the proof of Proposition~\ref{prop:onlyVanishing}.
\end{proof}

\subsection{The region \texorpdfstring{$\frac 12<\alpha<\frac 23$}{1/2<alpha<2/3}}\label{ssec:above}

We recall from Proposition \ref{nonnegative} the existence of a critical value $V_c$ such that $I_M=0$ if and only if $C\,M^{4\alpha-2}\le V_c$, and $I_M<0$ otherwise. Let us define
\[
M_c:=\left(\frac{V_c}C\right)^{\frac1{4\alpha-2}}
\]
and notice that $4\alpha-2$ is positive if $\frac 12<\alpha<\frac 23$. The main result in this region is stated in the following proposition.

\begin{proposition}\label{prop:above} Assume that $\alpha\in(\frac12,\frac23)$. The following assertions hold true:
\begin{enumerate}
\item[(i)] If $M<M_c$, then $I_M$ is not achieved.
\item[(ii)] If $M=M_c$, then there exists a minimizer.
\item[(iii)] If $M>M_c$, then the strict inequalities \eqref{ineqstrict} always hold, and in particular there exists a minimizer.
\end{enumerate}
\end{proposition}

In the critical case $M=M_c$, the strict inequalities \eqref{ineqstrict} do not hold. As consequence, the stability of such a solution cannot be ensured by usual arguments.

\begin{proof} We first assume that $M<M_c$, so that $I_M=0$ by Proposition~\ref{nonnegative}. Define
\[
\E_M[\varphi]:=\frac12\int_{\Real^3}|\nabla\varphi|^2\,dx+\,\frac14\,\D\varphi-M^{4\alpha-2}\,\frac C{2\alpha+2}\int_{\Real^3}|\varphi|^{2\alpha+2}\,dx\,.
\]
We may observe that $\E_1=\E$. By applying Lemma~\ref{scaling} with $p=2$ and $q=1$ (or, equivalently, \eqref{eq:enerscal2} with $p=2$), we get
\[
\E\big[M^2\,\varphi(M\cdot)\big]=M^3\,\E_M[\varphi]\quad\forall\,\varphi\in\Sigma_1\,.
\]
We argue by contradiction. Assume that $I_M$ is achieved. Then, there exists a minimizer $\varphi_M$ of
\[
\inf\{\E_M[\varphi]\,:\,\varphi\in\Sigma_1\}=M^{-3}\,I_M\,.
\]
In this way, $\varphi_M$ is a test function for $M_c^{-3}\,I_{M_c}$ and $\E_{M_c}[\varphi_M]<\E_M[\varphi_M]=0$ since $M_c>M$. We contradict the fact that $I_{M_c}=0$, thus proving (i).

\medskip Next, we assume that $M>M_c$. In order to prove the strict inequalities \eqref{ineqstrict} and establish (iii) in Proposition~\ref{prop:above}, the key point is the following.

\begin{lemma}\label{lem:keyscaling} If $M> M_c$, then we have
\be{keyscaling}
I_{M'}\le\Big(\frac{M'}M\Big)^3\,I_M\quad\forall\,M'>M\,.
\ee
In particular, the function $M\mapsto I_M$ is decreasing on $[M_c,+\infty)$. Furthermore,
\be{strictm}
I_M<I_m+I_{M-m}\quad\forall\,m\in(0,M)\,.
\ee
\end{lemma}

\begin{proof} Consider $\varphi\in\Sigma_M$ and let $\tilde\varphi:=\big(\frac{M'}M \big)^2\,\varphi\big(\frac{M'}M\cdot\big)$. We notice that $\tilde\varphi\in\Sigma_{M'}$~and
\begin{eqnarray*}
I_{M'}\le\E[\tilde\varphi]=\left(\tfrac{M'}M\right)^3\,\left[\tfrac12\int_{\Real^3}|\nabla\varphi|^2\,dx+\,\tfrac12\,\D\varphi \right.-\left.\tfrac C{2\alpha+2}\,\left(\tfrac{M'}M\right)^{4\alpha-2}\int_{\Real^3}|\varphi|^{2\alpha+2}\,dx\right]&&\\
\le\left(\tfrac{M'}M\right)^3\,\E[\varphi]\,.&&
\end{eqnarray*}
We deduce \eqref{keyscaling} by taking the infimum of the right-hand side over all functions $\varphi$ in $\Sigma_M$ and the monotonicity of $M\mapsto I_M$ on $[M_c,+\infty)$ follows.

We now turn our attention to the proof of \eqref{strictm}. If $I_m=I_{M-m}=0$, inequalities \eqref{strictm} obviously holds for any $M>M_c$, since $I_M$ is negative. If $I_m<0$ but $I_{M-m}=0$ (so that $M_c<m$ and $M-m\le M_c$), then \eqref{strictm} reduces to $I_M<I_m$, which is again guaranteed by \eqref{keyscaling}. If both $I_m$ and $I_{M-m}$ are negative (this is equivalent to $m>M_c$ and $M-m>M_c$, and therefore it may occur only if $M>2M_c$), then we have
\[
I_M\le\Big(\frac Mm\Big)^3\,I_M<\frac Mm\,I_m\quad\mbox{and}\quad I_M\le\Big(\frac M{M-m}\Big)^3\,I_M<\frac M{M-m}\,I_{M-m}
\]
by using \eqref{keyscaling}. Hence, $I_M=\frac m M\,I_M+\frac {M-m}M\,I_M<I_m+I_{M-m}$. This concludes the proof of Lemma~\ref{lem:keyscaling}.\end{proof}

Let us come back to the proof of Proposition \ref{prop:above}. In order to prove the existence of minimizers in the limiting case $C\,M^{4\alpha-2}=V_c$, that is $M=M_c$, we follow the arguments in~\cite{JeanLuo2012}, where a proof for the case $C=1$ is given. As noted in Remark~\ref{Rem:Vanishing}, relative compactness (up to translations) of all minimizing sequences cannot be proved in this case, since $I_{M_c}=0$. We build a particular minimizing sequence as follows.

Let $M_n=M_c+\frac1n$, for every positive integer $n$, and assume that $\varphi_n$ is a minimizer of $I_{M_n}$ in $\Sigma_{M_n}$, which is already known to exist since $M_n>M_c$ and therefore $I_{M_n}<0$ for any $n\ge1$. Since $\{M_n\}_{n\ge1}$ converges towards $M_c$, it can be deduced that $\lim_{n\to\infty}\E[\varphi_n]=\lim_{n\to\infty}I_{M_n}=I_{M_c}=0$. With the notations of Corollary~\ref{Cor:CriticalMass}, this means that $\lim_{n\to\infty}\eta_{M_n}=0$. If we combine the results of Corollary~\ref{cor:EstimCrit} and Corollary~\ref{Cor:CriticalMass}, then we obtain
\begin{multline*}
\frac 14\,\varepsilon_{M_n}-\frac 32\,\eta_{M_n}\le\C_{1/2}^{2-2\alpha}\,\CGN(1)^{2\alpha-1}\,M^{\alpha-\frac 12}\,\left[\frac 12\,(3\alpha-1)\,\varepsilon_{M_n}-(5\alpha-1)\,\eta_{M_n}\right] ^{4\alpha-1}\\
\Big[(2-3\alpha)\,\varepsilon_{M_n}-2\,(2-\alpha)\,\eta_{M_n}\Big]^{1-\alpha}\,.
\end{multline*}
By passing to the limit as $n\to\infty$, we find that
\[
\frac 14\le\C_{1/2}^{2-2\alpha}\,\CGN(1)^{2\alpha-1}\,M^{\alpha-\frac 12}\,\left[\frac 12\,(3\alpha-1)\right] ^{4\alpha-1}\Big[(2-3\alpha)\Big]^{1-\alpha}\liminf_{n\to\infty}\varepsilon_{M_n}^{3\alpha-1}\,,
\]
thus proving that $\liminf_{n\to\infty}\varepsilon_{M_n}>0$ and hence
\[
\liminf_{n\to\infty}\ir{|\varphi_n|^{2\alpha+2}}>0\,.
\]
Then, by Lemma I.1 in~\cite{bi:PLL-CC2} the sequence $\{\varphi_n\}_{n\ge 1}$ satisfies the non-vanishing condition \eqref{nonvanishing}. Consequently, up to translations, there exists a subsequence that converges weakly in $\Huno$, strongly in $\L^2_{\rm loc} (\Real^3)$ and pointwise almost everywhere, towards a nonzero function $\varphi_\infty$. This sequence can be also assumed to be strongly convergent in $\L^2_{\rm loc} (\Real^3)$ and pointwise convergent almost everywhere. Thanks to~\cite{bi:BrL1} and Lemma 2.2 in~\cite{ZaZa}, we get
\[
0=\lim_ {n \to \infty}I_{M_n}=\lim_ {n \to \infty}\E[\varphi_n]=\E[\varphi_\infty]+\lim_ {n \to \infty}\E[\varphi_n-\varphi_\infty]\,.
\]
Since $0 <\|\varphi_\infty \|_{\Ldos}^2\le M_c$, we have $\E[\varphi_\infty]\geq 0$ and
\[
\lim_ {n \to \infty} \| \varphi_n-\varphi_\infty \|_{\Ldos}^2=M_c-\|\varphi_\infty \|_{\Ldos}^2 <M_c\,,
\]
then $\lim_ {n \to \infty}\E[\varphi_n-\varphi_\infty]\ge 0$. Therefore, $\E[\varphi_\infty]=0$. To conclude the proof of Proposition \ref{prop:above} we observe that $\|\varphi_\infty \|_{\Ldos}^2=M_c$. Otherwise, $\varphi_\infty$ is a minimizer of $I_M $, for some $M<M_c$, and we reach a contradiction with the first statement in Proposition \ref{prop:above}.\hfill$\square$\end{proof}


\medskip\noindent{\scriptsize\copyright~2013 by the authors. This paper may be reproduced, in its entirety, for non-commercial~purposes.}

\bibliographystyle{siam}\small\bibliography{References}
\end{document}